\documentclass[a4paper,12pt]{article}
\usepackage{a4wide}
\usepackage{amsmath}
\usepackage{amssymb}
\usepackage{amsthm}
\usepackage{latexsym}
\usepackage{graphicx}
\usepackage[english]{babel}
\usepackage{makeidx}
\usepackage[mathlines]{lineno}

\newtheorem{obs} [subsection]{Remark}
\newtheorem{exm} [subsection]{Example}

\newtheorem{conj}[subsection]{Conjecture}
\newtheorem{teor}[subsection]{Theorem}
\newtheorem{lema}[subsection]{Lemma}
\newtheorem{cor} [subsection]{Corollary}
\newcommand{\Zng}{$\mathbb Z^n$-graded $S$-module}

\def\sdepth{\operatorname{sdepth}}
\def\depth{\operatorname{depth}}
\def\supp{\operatorname{supp}}
\def\deg{\operatorname{deg}}

\def\Ass{\operatorname{Ass}}

\begin{document}
\selectlanguage{english}
\frenchspacing

\numberwithin{equation}{section}

\title{Depth and Stanley depth of powers of the path ideal of a cycle graph}
\author{Silviu B\u al\u anescu$^1$ and Mircea Cimpoea\c s$^2$}
\date{}

\maketitle

\footnotetext[1]{ \emph{Silviu B\u al\u anescu}, University Politehnica of Bucharest, Faculty of
Applied Sciences, 
Bucharest, 060042, E-mail: silviu.balanescu@stud.fsa.upb.ro}
\footnotetext[2]{ \emph{Mircea Cimpoea\c s}, University Politehnica of Bucharest, Faculty of
Applied Sciences, 
Bucharest, 060042, Romania and Simion Stoilow Institute of Mathematics, Research unit 5, P.O.Box 1-764,
Bucharest 014700, Romania, E-mail: mircea.cimpoeas@upb.ro,\;mircea.cimpoeas@imar.ro}

\begin{abstract}
Let $J_{n,m}:=(x_1x_2\cdots x_m,\;  x_2x_3\cdots x_{m+1},\; \ldots,\; x_{n-m+1}\cdots x_n,\; x_{n-m+2}\cdots x_nx_1, \linebreak
\ldots, x_nx_1\cdots x_{m-1})$ be the $m$-path ideal of the cycle graph of length $n$, in the ring $S=K[x_1,\ldots,x_n]$.

Let $d=\gcd(n,m)$. We prove that $\depth(S/J_{n,m}^t)\leq d-1$ for all $t\geq n-1$.
We show that $\sdepth(S/J_{n,n-1}^t)=\depth(S/J_{n,n-1}^t)=\max\{n-t-1,0\}$ for all $t\geq 1$.
Also, we give some bounds for $\depth(S/J_{n,m}^t)$ and $\sdepth(S/J_{n,m}^t)$, where $t\geq 1$.


\end{abstract}

\section*{Introduction}


Let $K$ be a field and $S=K[x_1,\ldots,x_n]$ the polynomial ring over $K$. The study of the edge ideals associated to graphs is a classical topic in 
combinatorial commutative algebra.
Conca and De Negri generalized the definition of an edge ideal and first introduced the 
notion of a $m$-path ideal in \cite{conca}. In the recent years, several algebraic and combinatorial 
properties of path ideals have been studied extensively. However, little is known about the powers
of $m$-path ideals. 

Following our previous work \cite{lucrare1}, the aim of our paper is to investigate 
the $\depth$ and the Stanley depth ($\sdepth$) of the quotient rings associated to powers of the $m$-path ideal of a cycle.
For the definition of the $\sdepth$ invariant see Section $2$.

For $n\geq m\geq 1$, the \emph{$m$-path ideal of the path graph} of length $n$ is
$$I_{n,m}=(x_1x_2\cdots x_m,\;  x_2x_3\cdots x_{m+1},\; \ldots,\; x_{n-m+1}\cdots x_n)\subset S.$$
The \emph{$m$-path ideal of the cycle graph} of length $n$ is
$$J_{n,m}=I_{n,m}+(x_{n-m+2}\cdots x_nx_1,x_{n-m+3}\cdots x_nx_1x_2,\ldots,x_nx_1\cdots x_{m-1}).$$

 In \cite{lucrare1} we proved that
 $$\depth(S/I_{n,m}^t) =\varphi(n,m,t):= \begin{cases} n -t+2 -
 \left\lfloor \frac{n-t+2}{m+1} \right\rfloor - \left\lceil \frac{n-t+2}{m+1} \right\rceil, & t \leq n+1-m
 \\ m-1,& t > n+1-m \end{cases}.$$
 Also, we prove that $\sdepth(S/I_{n,m}^t) \geq \depth(S/I_{n,m}^t)$ and $\sdepth(I_{n,m}^t)\geq \depth(I_{n,m}^t)$.

The scope of our paper is to obtain similar results for powers of the ideal $J_{n,m}$. Let $n>m\geq 2$ and $t\geq 1$.
For $m=2$, Mihn, Trung and Vu \cite{mtv} prove that
$$\depth(S/J_{n,2}^t)=\left\lceil \frac{n-t+1}{3} \right\rceil\text{ for all }2\leq t < \left\lceil \frac{n+1}{2} \right\rceil.$$
Let $d=\gcd(n,m)$ and let $t_0\leq n-1$ be maximal with the property that there exists an integer $\alpha$ such that 
$mt_0 = \alpha n + d$. In Theorem \ref{t2} we prove that if $d=1$ then 
$$\sdepth(S/J_{n,m}^t) = \depth(S/J_{n,m}^t) = 0 \text{ for all }t\geq t_0.$$
Also, we prove that if $d>1$ then 
$$\depth(S/J_{n,m}^t) \leq d-1 \text{ and } \sdepth(S/J_{n,m}^t) \leq n - \frac{n}{d} \text{ for all }t\geq t_0.$$
In Corollary \ref{t4} we prove that if $n$ is odd, then
$$\sdepth(S/J_{n,n-2}^t)=\depth(S/J_{n,n-2}^t)=0\text{ for all }t\geq \frac{n-1}{2}.$$
Also, we prove that if $n$ is even, then
$$\depth(S/J_{n,n-2}^t)\leq 1 \text{ and }\sdepth(S/J_{n,n-2}^t) \leq \frac{n}{2}\text{ for all }t\geq n-1.$$
In Theorem \ref{t3}, we prove that $$\depth(S/J_{n,m}^t)\leq \varphi(n-1,m,t)+1.$$
In Theorem \ref{t1} we prove that
$$\sdepth(S/J_{n,n-1}^t)=\depth(S/J_{n,n-1}^t)=\begin{cases} n -t -1, & t \leq n-1 \\ 0,& t \geq n \end{cases}.$$
In Theorem \ref{t5}, we show that if $n=mt-1$ then 
$$\sdepth(S/J_{n,m}^s)=\depth(S/J_{n,m}^s)=0\text{ for all }s\geq t.$$
Also, for $n\geq mt$, we prove that $$\sdepth(S/J_{n,m}^t),\depth(S/J_{n,m}^t)\geq \varphi(n-1,m,t).$$

\section{Preliminaries}

First, we recall the well known Depth Lemma, see for instance \cite[Lemma 2.3.9]{real}. 

\begin{lema}\label{l11}(Depth Lemma)
If $0 \rightarrow U \rightarrow M \rightarrow N \rightarrow 0$ is a short exact sequence of modules over a local ring $S$, or a Noetherian graded ring with $S_0$ local, then
\begin{enumerate}
\item[(1)] $\depth M \geq \min\{\depth N,\depth U\}$.
\item[(2)] $\depth U \geq \min\{\depth M,\depth N +1\}$.
\item[(3)] $\depth N \geq \min\{\depth U-1,\depth M\}$.
\end{enumerate}
\end{lema}

Let $M$ be a \Zng. A \emph{Stanley decomposition} of $M$ is a direct sum $\mathcal D: M = \bigoplus_{i=1}^r m_i K[Z_i]$ as a $\mathbb Z^n$-graded $K$-vector space, where $m_i\in M$ is homogeneous with respect to $\mathbb Z^n$-grading, $Z_i\subset\{x_1,\ldots,x_n\}$ such that $m_i K[Z_i] = \{um_i:\; u\in K[Z_i] \}\subset M$ is a free $K[Z_i]$-submodule of $M$. We define $\sdepth(\mathcal D)=\min_{i=1,\ldots,r} |Z_i|$ and $\sdepth(M)=\max\{\sdepth(\mathcal D)|\;\mathcal D$ is a Stanley decomposition of $M\}$. The number $\sdepth(M)$ is called the \emph{Stanley depth} of $M$. 

Herzog, Vladoiu and Zheng show in \cite{hvz} that $\sdepth(M)$ can be computed in a finite number of steps if $M=I/J$, where $J\subset I\subset S$ are monomial ideals. In \cite{rin}, Rinaldo give a computer implementation for this algorithm, in the computer algebra system $\mathtt{CoCoA}$ \cite{cocoa}. We say that a \Zng $\;M$ satisfies the Stanley inequality, if 
$$\sdepth(M)\geq \depth(M).$$
In \cite{apel}, J.\ Apel restated a conjecture firstly given by Stanley in \cite{stan}, namely that any \Zng $\;M$ satisfies the Stanley
inequality. This conjecture proves to be false, in general, for $M=S/I$ and $M=J/I$, where $0\neq I\subset J\subset S$ are monomial ideals, see \cite{duval}, but remains open for $M=I$. 

The explicit computation of the Stanley depth it is a difficult task, even in very particular cases, and it is interesting in itself. 
Also,  although the Stanley conjecture was disproved in the most general set up, it is interesting to find large classes of ideals which satisfy the Stanley inequality. For a friendly introduction in the thematic of Stanley depth, we refer the reader \cite{her}.

In \cite{asia}, Asia Rauf proved the analog of Lemma \ref{l11} for $\sdepth$:

\begin{lema}\label{asia}
If $0 \rightarrow U \rightarrow M \rightarrow N \rightarrow 0$ is a short exact sequence of $\mathbb Z^n$-graded $S$-modules, then
$\sdepth(M) \geq \min\{\sdepth(U),\sdepth(N) \}$.
\end{lema}

We recall the following well known result (see for instance \cite[Lemma 2.3.10]{real}):

\begin{lema}\label{freg}
Let $M$ be a graded $S$-module and $f\in\mathfrak m=(x_1,\ldots,x_n)\subset S$ a homogeneous polynomial such that $f$ is regular on $M$.
Then $\depth(M/fM)=\depth(M)-1$.
\end{lema}

We also recall the following well known results. See for instance \cite[Corollary 1.3]{asia}, \cite[Proposition 2.7]{mirci},
\cite[Theorem 1.1]{mir}, \cite[Lemma 3.6]{hvz} and \cite[Corollary 3.3]{asia}.

\begin{lema}\label{lem}
Let $I\subset S$ be a monomial ideal and let $u\in S$ a monomial such that $u\notin I$. Then
\begin{enumerate}
\item[(1)] $\sdepth(S/(I:u))\geq \sdepth(S/I)$.
\item[(2)] $\depth(S/(I:u))\geq \depth(S/I)$.
\end{enumerate}
\end{lema}

\begin{lema}\label{lemm}
Let $I\subset S$ be a monomial ideal and let $u\in S$ a monomial such that $I=u(I:u)$. Then
  \begin{enumerate}
  \item[(1)] $\sdepth(S/(I:u))=\sdepth(S/I)$.
	\item[(2)] $\depth(S/(I:u))=\depth(S/I)$.
  \end{enumerate}
\end{lema}	
	
	
\begin{lema}\label{lhvz}
Let $I\subset S$ be a monomial ideal and $S'=S[x_{n+1}]$. Then
\begin{enumerate}
\item[(1)] $\sdepth_{S'}(S'/IS')=\sdepth_S(S/I)+1$, 
\item[(2)] $\depth_{S'}(S'/IS')=\depth_S(S/I)+1$.
\end{enumerate}
\end{lema}

\begin{lema}\label{lem7}
Let $I\subset S$ be a monomial ideal. Then the following assertions are equivalent:
\begin{enumerate}
\item[(1)] $\mathfrak m=(x_1,\ldots,x_n)\in \Ass(S/I)$.
\item[(2)] $\depth(S/I)=0$.
\item[(3)] $\sdepth(S/I)=0$.
\end{enumerate}
\end{lema}

Let $2\leq m < n$ be two integers. We consider the ideal
$$I_{n,m}=(x_1\cdots x_m,\; x_2\cdots x_{m+1},\; \ldots,\; x_{n-m+1}\cdots x_n)\subset S.$$
We denote
$\varphi(n,m,t):= \begin{cases} n -t+2 -
 \left\lfloor \frac{n-t+2}{m+1} \right\rfloor - \left\lceil \frac{n-t+2}{m+1} \right\rceil, & t \leq n+1-m
 \\ m-1,& t > n+1-m \end{cases}.$

We recall the main result of \cite{lucrare1}:

\begin{teor}(See \cite[Theorem 2.6]{lucrare1})\label{depth}
With the above notation, we have that
\begin{enumerate}
\item[(1)] $\sdepth(S/I_{n,m}^t)\geq \depth(S/I_{n,m}^t)=\varphi(n,m,t),\text{ for any }1\leq m\leq n\text{ and }t\geq 1$.
\item[(2)] $\sdepth(S/I_{n,m}^t)\leq \sdepth(S/I_{n,m})=\varphi(n,m,1)$.
\end{enumerate}
\end{teor}

\section{Main results}

We consider the following ideal
$$J_{n,m}=I_{n,m}+(x_{n-m+2}\cdots x_nx_1,\; x_{n-m+3}\cdots x_nx_1x_2,\; \ldots,\; x_nx_1\cdots x_{m-1}).$$
Let $d=\gcd(n,m)$ and let $t_0:=t_0(n,m)$ be the maximal integer such that $t_0\leq n-1$ and there exists a positive
integer $\alpha$ such that 
$$mt_0 = \alpha n + d.$$
Let $t\geq t_0$ be an integer.
Let $w=(x_1x_2\cdots x_n)^{\alpha}$, $w_t=w\cdot (x_1\cdots x_m)^{t-t_0}$, $r:=\frac{n}{d}$ and $s:=\frac{m}{d}$.
If $d>1$, we consider the ideal
$$U_{n,d}=(x_1,x_{d+1},\cdots,x_{d(r-1)+1})\cap (x_2,x_{d+2},\cdots,x_{d(r-1)+2})\cap \cdots \cap (x_d,x_{2d},\ldots,x_{rd}).$$
Firstly, we state the following lemma:

\begin{lema}\label{bije}
The map $\frac{\mathbb Z/n\mathbb Z}{r\cdot \left( \mathbb Z/n\mathbb Z \right)} \stackrel{\cdot s}{\longrightarrow}
\frac{\mathbb Z/n\mathbb Z}{r\cdot \left( \mathbb Z/n\mathbb Z \right)} $ is bijective.
\end{lema}

\begin{proof}
It follows from the fact that $\gcd(s,r)=1$.
\end{proof}

\begin{lema}\label{lemmy}
With the above notations, we have:
\begin{enumerate}
\item[(1)] If $d=1$ then $(J_{n,m}^{t}:w_t)=\mathfrak m$ for all $t\geq t_0$.
\item[(2)] If $d>1$ then $(J_{n,m}^{t}:w_t)=U_{n,d}$ for all $t\geq t_0$.
\end{enumerate}
\end{lema}

\begin{proof}
(1) Note that $\widehat{t_0} = \widehat{m}^{-1}$ in $\mathbb Z/n\mathbb Z$, hence $t_0$ and $\alpha$ are uniquely defined.
    We claim that is enough to show the assertion for $t=t_0$, that is $(J_{n,m}^{t_0}:w)=\mathfrak m$.
		
		Assume that $(J_{n,m}^{t_0}:w)=\mathfrak m$ and $t>t_0$. Since $x_jw\in J_{n,m}^{t_0}$ for all $1\leq j\leq n$, it follows
		that $x_jw_t=x_jw(x_1\cdots x_m)^{t-t_0}\in J_{n,m}^t$ for all $1\leq j\leq n$, and therefore $\mathfrak m\subset (J_{n,m}^t:w_t)$.
		On the other hand, $w_t\notin J_{n,m}^t$ since $\deg(w_t)=mt-1$ and $J_{n,m}^t$ is generated in degree $mt$. Hence
		$(J_{n,m}^{t}:w_t)=\mathfrak m$, and the claim is proved.

    Since $J_{n,m}$ is invariant to circular permutations of variables and $w\notin J_{n,m}^{t_0}$, it is enough to show that
    \begin{equation}\label{clem}
      x_1w=x_1^{\alpha+1}x_2^{\alpha}\cdots x_n^{\alpha} \in G(J_{n,m}^{t_0}).
    \end{equation}
		Indeed, one can easily check that 
		$$x_1w=\prod_{j=0}^{t_0-1}(x_{\ell(mj+1)}\cdots x_{\ell(mj+m)}),$$ where
    $\ell(k)\in \{1,\ldots,n\}$ is the unique integer with $k\equiv \ell(k)(\bmod\;n)$. 

    As $x_{\ell(mj+1)}\cdots x_{\ell(mj+m)}\in G(J_{n,m})$ for all $0\leq j\leq t_0-1$, we proved \eqref{clem} and thus $(1)$.
		
(2) Note that $\deg(w_t)=\alpha n+m(t-t_0) = mt - d$, while $J_{n,m}^t$ is minimally generated by monomials of degree $mt$.
    Also, as $w_t=(x_1\cdots x_m)^{t-t_0}w$ and $x_1\cdots x_m\in G(J_{n,m})$, we have that
		\begin{equation}\label{ek1}
		(J_{n,m}^{t_0}:w) \subseteq (J_{n,m}^t:w_t).
		\end{equation}
	  Let $u=x_1^{a_1}x_2^{a_2}\cdots x_n^{a_n}\in G(J_{n,m}^t)$. 
		Note that, if $u\in G(J_{n,m})$ then $\supp(u)$ contains exactly $s=\frac{m}{d}$ variables whose indices are congruent with $j$ 
		modulo $d$, where $0\leq j\leq d-1$. Therefore, as $d=\gcd(n,m)$, it follows that
		\begin{equation}\label{cooroo}
		a_1+a_{d+1}+\cdots+a_{d(r-1)+1}=a_2+a_{d+2}+\cdots+a_{d(r-1)+2}=\cdots =a_d+a_{2d}+\cdots+a_{rd}=\frac{tm}{d}.
		\end{equation}
		Similarly, if we rewrite $w_t=(x_1\cdots x_n)^{\alpha}(x_1\cdots x_m)^{t-t_0}$ as $w_t=x_1^{b_1}x_2^{b_2}\cdots x_n^{b_n}$
		then we have
		\begin{equation}\label{cooroocoo}
		b_1+b_{d+1}+\cdots+b_{d(r-1)+1}=\cdots =b_d+b_{2d}+\cdots+b_{rd}= \frac{n\alpha}{d}+\frac{m(t-t_0)}{d}= \frac{tm}{d}-1.
		\end{equation}	
    Let $v\in S$ be a monomial such that $vw_t\in J_{n,m}^t$. It follows that there exists $u\in G(J_{n,m}^t)$ such that $u|vw_t$. 
    From \eqref{cooroo} and \eqref{cooroocoo} it follows that
    for every $0\leq j\leq d-1$ there exists $k_j\in \{1,\ldots,n\}$
		with $k_j\equiv j(\bmod\; d)$ such that $x_{k_j}|v$. Therefore, $v\in U_{n,d}$ and thus
		$(J_{n,m}^t:w_t)\subseteq U_{n,d}$.
		Hence, from \eqref{ek1}, in order to prove that $(J_{n,m}^t:w_t) = U_{n,d}$, it is enough to show that 
		\begin{equation}\label{ek2}
		U_{n,d} \subseteq (J_{n,m}^{t_0}:w).
		\end{equation}
    Let $v=x_{\ell_1}x_{\ell_2}\cdots x_{\ell_d}\in G(U_{n,d})$, where $\ell_j \equiv j (\bmod\;d)$.
		In order to prove \eqref{ek2},
		it suffices to show that $vw\in G(J_{n,m}^{t_0})$.
		
		Since $mt_0=\alpha n +d$, by dividing with $d$, it follows that $st_0=\alpha r +1$ and therefore		
		$\overline{t_0}=\overline{s}^{-1}$ in $\mathbb Z/r\mathbb Z$. If $\overline{t_0}=\overline{\ell}$ with $0\leq \ell\leq r-1$,
		then we claim that \begin{equation}\label{tz} t_0=\ell+n-r \geq n-r = r(d-1).\end{equation}
		Let $t_0'=\ell+n-r$. Since $\overline{t_0'}=\overline{t_0}=\overline{s}^{-1}$, we can write
    $\ell s=\beta r + 1$ for some $\beta$ and thus 
		$$st_0'=s(\ell +n-r) = \beta r + s(n-r) + 1 = (\beta+ s(d-1))r + 1 = \alpha' r +1,\text{ where }\alpha'=\beta+s(d-1).$$
		 Hence $mt_0'=\alpha' r +d$. Since $\overline{t_0'}=\overline{t_0}$ in $\mathbb Z/r\mathbb Z$
		and $t_0\leq n-1$ is the greatest integer with $mt_0=\alpha r +d$, it follows that $t_0=t_0'$. Hence, we proved \eqref{tz}.
	 
	  For simplicity, if $j>n$, we denote by $x_j$ the variable $x_{\ell(j)}$,
		where $1\leq \ell(j)\leq n$ such that $j\equiv \ell(j)(\bmod\;n)$. See also the proof of (1).
		
		Given a monomial $u=x_ix_{i+1}\cdots x_{i+m-1}\in G(J_{n,m})$, we let $x_{\min(u)}=x_i$ and $x_{\max(u)}=x_{i+m-1}$ (with the above convention).
		
		We apply the following algorithm:
		\begin{enumerate}
		\item We let $u_1:=x_{\ell_d}x_{\ell_d+1}\cdots x_{\ell_d+m-1}$, where $v=x_{\ell_1}\cdots x_{\ell_d}$; see above.
		\item Assume we defined $u_1,\ldots,u_k$, where $1\leq k\leq t_0-1$. 
		
		If $x_{\max(u_k)}=x_{\ell_j}$ for some $1\leq j\leq d-1$, 
		then we let $u_{k+1}=x_{\ell_j}x_{\ell_j+1}\cdots x_{\ell_j+m-1}$.
		Otherwise, we let $u_{k+1}:=x_{\max(u_k)+1}\cdots x_{\max(u_k)+m}$.
		\item We repeat the step 2. until $k=t_0$.
		\end{enumerate}
		We claim that $vw=u_1u_2\cdots u_{t_0}$. Obviously, $\deg(vw)=mt_0=\deg(u_1\cdots u_{t_0})$.
		
		Let $k_1$ be the minimal index with $\max(u_{k_1})=d-1$. We claim that $k_1\leq r$. Indeed, if $k_1 \geq r$ then 
		$u_1=x_{\ell_d}\cdots x_{\ell_d+m-1},\ldots,u_{r}=x_{\ell_d+(r-1)m}\cdots x_{\ell_d + rm-1}$. From Lemma \ref{bije} and 
		the fact that $m=sd$, it
		follows that $$\{\ell_d+m-1,\cdots,\ell_d+rm-1\} = \{d-1,2d-1,\ldots,rd-1\}
		\text{ using the above convention.}$$
		Since $\ell_{d-1} \equiv (d-1)(\bmod\;d)$, from all the above, it follows that $\ell_{d-1} =\ell_d+rm-1$ and hence $k_1=r$.
		Note that $u_{k_1}=x_{\ell_{d-1}}\cdots x_{\ell_{d-1}+m-1}$.	
		
		Similarly, let $k_2$ be the minimal index with $\max(u_{k_2})=d-2$. Using the same line of arguing, it follows that $k_2\leq 2r$.
		Inductively, let $k_j$ be the minimal index with $\max(u_{k_j})=d-j$, for $j\leq d-1$. Then $k_j\leq jr$. In particular, we
		have that $k_{d-1}\leq (d-1)r \leq t_0$. Also, 		
		for $k>k_{d-1}$, from the definition of $u_k$'s, we have that $\max(u_k)\notin \{\ell_1,\ldots,\ell_{d-1}\}$.
		
		Now, from all the above, it is easy to see that $u_1\cdots u_{t_0}=vw$, as required.
\end{proof}

In the following example, we show how the algorithm given in the proof of Lemma \ref{lemmy}(2) works:

\begin{exm}\rm
Let $n=12$ and $m=8$. Then $d=\gcd(n,m)=4$, $r=3$ and $s=2$. Note that $8\cdot 11 = 7\cdot 12 + 4$ and $t_0=11$ is
the largest integer $\leq n-1=11$ with $t_0m=\alpha n+d$. Also $\alpha=7$. 

We have $w:=(x_1\cdots x_{12})^7$ and
$U_{12,4}=(x_1,x_5,x_9)\cap (x_2,x_6,x_{10})\cap (x_3,x_7,x_{11})\cap (x_4,x_8,x_{12})$. 

Let $v:=x_5x_{2}x_{11}x_4 \in U_{12,4}$. Then $\ell_1=5$, $\ell_2=2$, $\ell_3=11$ and $\ell_4=4$.

We apply the aforementioned algorithm:
\begin{itemize}
\item We let $u_1:=x_4x_5x_6x_7x_8x_9x_{10}x_{11}$.
\item Since $\ell_3=11$, it follows that $k_1=1$ and $u_2=x_{11}x_{12}x_1x_2x_3x_4x_5x_6$.
\item Since $\ell_2\neq 6$, we let $u_3=x_7x_8x_9x_{10}x_{11}x_{12}x_1x_2$.
\item Since $\ell_2=2$, it follows that $k_2=3$ and $u_4=x_2x_3x_4x_5x_6x_7x_8x_9$.
\item Since $\ell_1\neq 9$, we let $u_5=x_{10}x_{11}x_{12}x_1x_2x_3x_4x_5$.
\item Since $\ell_1=5$, it follows that $k_3=5$ and $u_6=x_5x_6x_7x_8x_9x_{10}x_{11}x_{12}$.
\end{itemize}
From now on, the algorithm goes smoothly, and we have: $u_7=x_1x_2x_3x_4x_5x_6x_7x_8$, $u_8=x_9x_{10}x_{11}x_{12}x_1x_2x_3x_4$,
$u_9=u_6$, $u_{10}=u_7$ and $u_{11}=u_8$.
It is easy to see that $u_1u_2\cdots u_{11}=vw$. Therefore $vw\in J_{12,8}^{11}$, as required.
\end{exm}

The following result is elementary. However, we give a proof in order of completeness.

\begin{lema}\label{liema}
Let $d\geq 1$ and $Z_1\cup Z_2\cup \cdots \cup Z_d=\{x_1,\ldots,x_n\}$ be a partition, i.e. $|Z_i|>0$ and $Z_i\cap Z_j=\emptyset$ for all $i\neq j$.
Let $P_i=(Z_i)\subset S$ for $1\leq i\leq d$ and $U:=P_1\cap \cdots \cap P_d$. Then $\depth(S/U)=d-1$.
\end{lema}

\begin{proof}
We use induction on $d\geq 1$. If $d=1$ then $Z_1=\{x_1,\ldots,x_n\}$ and $U=\mathfrak m=(x_1,\ldots,x_n)$. Hence, there is nothing to prove.

Without losing the generality, we can assume that $Z_1\cup\cdots \cup Z_{d-1}=\{x_1,\ldots,x_k\}$ for some $k<n$ and 
$Z_d=\{x_{k+1},\ldots,x_n\}$.
From induction hypothesis, we have that 
$$\depth(S_k/(P_1\cap \cdots \cap P_{d-1}))=d-2,\text{ where }S_k=K[x_1,\ldots,x_k].$$
From \cite[Lemma 1.1]{apop} it follows that \small
$$\depth(S/(P_1\cap\cdots \cap P_d))=\depth(S_k/(P_1\cap \cdots \cap P_{d-1}))+\depth(K[x_{k+1},\ldots,x_n]/P_d)+1=d-1,$$
\normalsize
as required.
\end{proof}

\begin{teor}\label{t2}
With the above notations, we have:
\begin{enumerate}
\item[(1)] If $d=1$ then $\sdepth(S/J_{n,m}^t) = \depth(S/J_{n,m}^t) = 0 \text{ for all }t\geq t_0$.
\item[(2)] If $d>1$ then $\depth(S/J_{n,m}^t) \leq d-1 \text{ for all }t\geq t_0$.
\item[(3)] If $d>1$ then $\sdepth(S/J_{n,m}^t) \leq \sdepth(S/U_{n,d})\leq n - \frac{n}{d}$ for all $t\geq t_0$.
\end{enumerate}
\end{teor}

\begin{proof}
(1) From Lemma \ref{lemmy}(1), it follows that $\mathfrak m \in \Ass(S/J_{n,m}^t)$ for all $t\geq t_0$.
    Therefore, the required conclusion follows from Lemma \ref{lem7}.
		
(2) From Lemma \ref{lemmy}(2), it follows that $(J_{n,m}^{t}:w_t)=U_{n,d}$. From Lemma \ref{lem}(2) it follows
that $$\depth(S/J_{n,m}^t)\leq \depth(S/U_{n,d}).$$
On the other hand, from Lemma \ref{liema} it follows that $\depth(S/U_{n,d})=d-1$ and therefore $\depth(S/J_{n,m}^t)\leq d-1$.

(3) Similarly, from Lemma \ref{lem}(1) it follows that $\sdepth(S/J_{n,m}^t)\leq \sdepth(S/U_{n,d})$.
On the other hand, since $U_{n,d}=(x_d,\ldots,x_{dr})\cap U'$, where $U'=(x_1,\ldots,x_{r(d-1)+1})\cap \cdots
\cap (x_{d-1},\ldots,x_{dr-1})$, from \cite[Theorem 1.3]{mirci}
it follows that $$\sdepth(S/U_{n,d})\leq \sdepth(S/(x_d,\ldots,x_{dr}))= n-r=n-\frac{n}{d},$$
as required.
\end{proof}

Our computer experiments in CoCoA \cite{cocoa} yield us to propose the following conjecture:

\begin{conj}\label{wish}
We have that $$\depth(S/J_{n,m}^t) \geq d-1\text{ for all }t\geq 1.$$
\end{conj}

\begin{obs}\rm
Let $n>m\geq 2$ be two integers and let $d:=\gcd(n,m)$. From Theorem \ref{t2} we have that
$\depth(S/J_{n,m}^t) \leq d-1$ for all $t\geq t_0$. Hence, if Conjecture \ref{wish} is true,
then $$\lim_{t\to\infty} \depth(S/J_{n,m}^t) = d-1.$$
\end{obs}

\begin{cor}\label{t4}
We have that:
\begin{enumerate}
\item[(1)] If $n$ is odd, then $\sdepth(S/J_{n,n-2}^t)=\depth(S/J_{n,n-2}^t)=0$ for all $t\geq \frac{n-1}{2}$.
\item[(2)] If $n$ is even, then $\depth(S/J_{n,n-2}^t)\leq 1$ for all $t\geq n-1$.
\item[(3)] If $n$ is even, then $\sdepth(S/J_{n,n-2}^t)\leq \frac{n}{2}$ for all $t\geq n-1$.
\end{enumerate}
\end{cor}

\begin{proof}
(1) Since $n$ is odd, we have $d=\gcd(n,n-2)=1$. It is easy to see that $t_0=\frac{n-1}{2}$ and $\alpha=\frac{n-1}{3}$.
    Hence, from Theorem \ref{t2}(1) it follows that 
		$$\sdepth(S/J_{n,n-2}^t)=\depth(S/J_{n,n-2}^t)=0\text{ for all }t\geq \frac{n-1}{2}.$$
		
(2) Since $n$ is even, we have $d=\gcd(n,n-2)=2$. It is easy to see that $t_0=n-1$ and $\alpha=n-3$.
    From Theorem \ref{t2}(2) it follows that 
		$$\depth(S/J_{n,n-2}^t)= 1\text{ for all }t\geq n-1.$$
	
(3) As in the proof of (2), from Theorem \ref{t2}(3) it follows that
		$$\sdepth(S/J_{n,n-2}^t)\leq \frac{n}{2}\text{ for all }t\geq n-1.$$
		Hence, the proof is complete.
\end{proof}

\begin{lema}\label{inmt}
Let $n>m\geq 2$ and $t\geq 1$ be some integers. Then
$$ (J_{n,m}^t,x_n)=(I_{n-1,m}^t,x_n).$$
\end{lema}

\begin{proof}
The inclusion $\supseteq$ is obvious. The converse inclusion follows from the observation that a minimal monomial generator of $J_{n,m}$ which
is not divisible by $x_n$ belongs to $I_{n-1,m}$
\end{proof}

\begin{teor}\label{t3}
Let $n > m\geq 2$ and $t\geq 1$ be some integers. Then:
\begin{enumerate}
\item[(1)] $\depth(S/J_{n,m}^t)\leq \varphi(n-1,m,t)+1$.
\item[(2)] If $\depth(S/(J_{n,m}^t:x_n))>\depth(S/J_{n,m}^t)$ then $\depth(S/J_{n,m}^t)=\varphi(n-1,m,t)$.
\end{enumerate}
\end{teor}

\begin{proof}
(1) We consider the short exact sequence
\begin{equation}\label{sire}
 0 \to S/(J_{n,m}^t:x_n) \to S/J_{n,m}^t \to S/(J_{n,m}^t,x_n) \to 0.
\end{equation}
From Lemma \ref{inmt} and Theorem \ref{depth} it follows that
\begin{equation}
\depth(S/(J_{n,m}^t,x_n)) = \depth(S/(I_{n-1,m}^t,x_n))=\varphi(n-1,m,t).
\end{equation}
Let $s:=\depth(S/(J_{n,m}^t:x_n))$ and $d:=\depth(S/J_{n,m}^t)$. From Lemma \ref{lem}(2), we have that $s\geq d$.
Therefore, according to Lemma \ref{l11} (Depth lemma), it follows that
$$ \varphi(n-1,m,t)\geq \min\{s-1,d\}\geq \min\{d-1,d\} = d-1. $$
Hence, $d\leq \varphi(n-1,m,t)+1$, as required.

(2) As in the proof of (1), it follows from \eqref{sire} and Lemma \ref{l11} (Depth lemma).
\end{proof}

\section{Some special cases}

We use the notations from the previous section.

\begin{teor}\label{t1}
We have that:
 $$\sdepth(S/J_{n,n-1}^t)=\depth(S/J_{n,n-1}^t)=\begin{cases} n -t -1, & t \leq n-2 \\ 0,& t \geq n-1 \end{cases}.$$
\end{teor}

\begin{proof}
Since $d=\gcd(n,m)=\gcd(n,n-1)=1$, it follows that $t_0:=t_0(n,n-1)=n-1$, since $mt_0=(n-1)^2 = (n-2)n+1 =\alpha n + d$.
Therefore, according to Theorem \ref{t2}(1), the conclusion follows for $t\geq n-1$.
Now, assume $t\leq n-2$.

It $n=3$, then $t=1$ and it is an easy exercise to show that
$$\sdepth(S/J_{3,2})=\depth(S/J_{3,2})=1=n-t-1.$$
Now, assume $n\geq 4$ and $t\leq n-2$. We consider the ideals
$$L_j:=(J_{n,n-1}^t:x_n^j)\text{ for }0\leq j\leq t.$$
By straightforward computations, we have
\begin{equation}\label{lj}
L_j=J_{n,n-1}^{t-j}(J_{n-1,n-2}^jS)\text{ for }0\leq j\leq t\text{ and }
\end{equation}
\begin{equation}\label{ljxn}
(L_j,x_n) = ((x_1\cdots x_{n-1})^{t-j}(J_{n-1,n-2}^jS),x_n)\text{ for }0\leq j\leq t-1.
\end{equation}
We consider the short exact sequences
$$ 0 \to S/L_1 \to S/L_0 \to S/(L_0,x_n) \to 0 $$
$$0 \to S/L_2 \to S/L_1 \to S/(L_1,x_n) \to 0 $$
$$ \vdots $$
\begin{equation}\label{iegzacte}
0 \to S/L_t \to S/L_{t-1} \to S/(L_{t-1},x_n) \to 0.
\end{equation}
From \eqref{ljxn}, the induction hypothesis and Lemma \ref{lemm} it follows that
\begin{equation}\label{cucu1}
\sdepth(S/(L_j,x_n))=\depth(S/(L_j,x_n))=n-j-2\text{ for all }0\leq j\leq t-1.
\end{equation}
On the other hand, from \eqref{lj} we have $L_t=(J_{n-1,n-2}^tS)$. Hence, from induction hypothesis
and Lemma \ref{lhvz} it follows that
\begin{equation}\label{cucu2}
\sdepth(S/L_t)=\depth(S/L_t)=(n-1-t-1)+1 = n-1-t.
\end{equation}
From \eqref{cucu1},\eqref{cucu2} and the short exact sequences \eqref{iegzacte} we deduce
inductively that
\begin{equation}\label{cucu3}
\sdepth(S/L_j),\depth(S/L_j)\geq n-1-t\text{ for all }0\leq j\leq t-1.
\end{equation} 
On the other hand, from Lemma \ref{lem} we have
\begin{equation}\label{cucu4}
\depth(S/L_0)\leq \depth(S/L_t)\text{ and }\sdepth(S/L_0)\leq \sdepth(S/L_t).
\end{equation}
Since $L_0=J_{n,n-1}^t$, from \eqref{cucu2},\eqref{cucu3} and \eqref{cucu4} it follows
that $$\sdepth(S/J_{n,n-1}^t)=\depth(S/J_{n,n-1}^t)=n-t-1,$$
which completes the proof.
\end{proof}

\begin{obs}\rm
Let $S_{n+m-1}:=K[x_1,x_2,\ldots,x_{n+m-1}]$. We note that
\begin{equation}\label{izo}
 \frac{S_{n+m-1}}{(I_{n+m-1,m}^t,x_1-x_{n+1},x_2-x_{n+2},\ldots,x_{m-1}-x_{n+m-1})} \cong \frac{S}{J_{n,m}^t}.
\end{equation}
Assume that $m=n-1$. It is not difficult to see that
\begin{equation}\label{claim}
x_1-x_{n+1},x_2-x_{n+2},\ldots,x_{n-2}-x_{2n-2}\text{ is a short exact sequence on }S_{2n-2}/I_{2n-2,n-1}^t.
\end{equation}
From \eqref{claim}, Lemma \ref{freg}, \eqref{izo} and Theorem \ref{depth} it follows that
$$\depth(S/J_{n,n-1}^t)=\depth(S_{2n-2}/I_{2n-2,n-1}^t)=\varphi(2n-2,n-1,t)-n+2,$$
from where we easily deduce the required formula given in Theorem \ref{t1} for $\depth(S/J_{n,n-1}^t)$.
Unfortunately, this method is not useful in the computation of $\sdepth(S/J_{n,n-1}^t)$.

We also mention that the sequence $x_1-x_{n+1},x_2-x_{n+2},\ldots,x_{n-2}-x_{2n-2}$ is not regular when $n>m+1$,
therefore we cannot use \eqref{izo} in order to compute (or at least to give some bounds for) $\depth(S/J_{n,m}^t)$.
\end{obs}

\begin{lema}\label{lemoasa}
Let $m,t\geq 2$ and $n\geq mt-1$ be some integers and $L=(J_{n,m}^t:(x_1x_2\cdots x_{mt-1}))$. We have that:
\begin{enumerate}
\item[(1)] If $n=mt-1$ then $L=\mathfrak m=(x_1,x_2,\ldots,x_n)$.
\item[(2)] If $n=mt$ then $L=(x_m,x_{2m},\ldots,x_{mt})$.
\item[(3)] It $mt<n\leq m(t+1)$ then $L=(x_m,x_{2m},\ldots,x_{mt},x_n)$.
\item[(4)] If $n>m(t+1)$ then $L=(x_m,\ldots,x_{mt},x_n)+V$, where 
$$V=(x_{mt+1}\cdots x_{mt+m},x_{mt+2}\cdots x_{mt+m+1},\ldots,x_{n-m}\cdots x_{n-1})^t\subset K[x_{mt+1},\ldots,x_{n-1}].$$
Moreover, $V\cong I_{n-mt-1,m}^t$.
\end{enumerate}
\end{lema}

\begin{proof}
(1)
As in the proof of Lemma \ref{lemmy}, we use the convention 
$$j=n+j=2n+j=\cdots,\text{ for all }1\leq j\leq n.$$ 
We fix $1\leq i\leq n$ and we define inductively the monomials $u_1:=x_ix_{i+1}\cdots x_{i+m-1}$ and $u_k:=x_{m_k+1}x_{m_k+2}\cdots x_{m_k+m}$,
where $m_1=i$ and $m_k=m_{k-1}+m$, for $2\leq k\leq t$. Obviously, $u_k\in G(J_{n,m})$ for all $1\leq k\leq t$, thus 
$u_1u_2\cdots u_t\in G(J_{n,m}^t)$. On the other hand, it is easy to see that 
$$x_i\cdot (x_1x_2\cdots x_{mt-1})=u_1u_2\cdots u_t.$$
Therefore, $x_i\in L$. Since $i$ was arbitrarily chosen, it follows that $\mathfrak m\subset L$.
Obviously, $L\neq S$, since $x_1x_2\cdots x_{mt-1}\notin J_{n,m}^t$. Therefore $L=\mathfrak m$, as required.

(2) Similarly to (1), we can deduce that $x_m,x_{2m},\ldots,x_{mt}\in L$. For instance, we have
$$x_m\cdot (x_1x_2\cdots x_{mt-1})=(x_1\cdots x_m)(x_m\cdots x_{2m-1})(x_{2m}\cdots x_{3m-1})\cdots(x_{mt-m}\cdots x_{mt-1})\in J_{n,m}^t.$$
Also, it is easy to see that 
$x_j\notin L$ for any $j\notin\{m,2m,\ldots,mt\}$. Since $$(J_{n,m}^t,x_m,x_{2m},\ldots,x_{mt})=(x_m,\ldots,x_{mt}),$$
the conclusion follows immediately.

(3) The proof is similar to the proof of $(2)$, with the remark that $x_n\in L$, since
$$x_n(x_1x_2\cdots x_{mt-1})=(x_nx_1\cdots x_{m-1})(x_m\cdots x_{2m-1})\cdots (x_{(t-1)m}\cdots x_{tm-1})\in J_{n,m}^t.$$

(4) As in the previous cases, it is easy to see that $(x_m,\ldots,x_{mt},x_n)\subset L$ and $x_j\notin L$ for any
$j\notin\{m,\ldots,mt,n\}$. Also, using similar arguments as in the proof of Lemma \ref{inmt}, we deduce that
$$(J_{n,m}^t,x_m,x_{2m},\ldots,x_{mt}) = (x_m,\ldots,x_{mt},x_n)+V.$$
Hence, we get the required conclusion.
\end{proof}

Using the above lemma, we are able to prove the following result:

\begin{teor}\label{t5}
Let $m,t\geq 2$ and $n\geq mt-1$ be some integers. We have that:
\begin{enumerate}
\item[(1)] If $n=mt-1$ then $\sdepth(S/J_{n,m}^s)=\depth(S/J_{n,m}^s)=0$ for all $s\geq t$.
\item[(2)] If $n\geq mt$ then $\sdepth(S/J_{n,m}^t)\geq \varphi(n-1,m,t)$.
\item[(3)] If $n\geq mt$ then $\varphi(n-1,m,t)+1\geq \depth(S/J_{n,m}^t)\geq \varphi(n-1,m,t)$.
\end{enumerate}
\end{teor}

\begin{proof}
Assume $n=mt-1$. From Lemma \ref{lemoasa}(1) it follows that
\begin{equation}\label{ohio}
(J_{n,m}^t:w)=\mathfrak m,\text{ where }w:=x_1x_2\cdots x_n.
\end{equation}
Let $w_s=w\cdot (x_1\cdots x_m)^{s-t}$. Since $w_s\notin J_{n,m}^s$, from \eqref{ohio} it follows that
$$(J_{n,m}^s:w_s)=\mathfrak m.$$
Therefore, $\mathfrak m\in \Ass(S/J_{n,m}^s)$ and (1) follows from Lemma \ref{lem7}.

Now, assume $n\geq mt$. 
Let $L_0=J_{n,m}^t$, $L_j:=(L_0:x_1\cdots x_j)$, for $1\leq j\leq mt-1$, and $U_j=(L_{j-1},x_j)$, for $1\leq j\leq mt-1$.
We consider the short exact sequences
\begin{equation}\label{cia1}
0 \to S/L_j \to S/L_{j-1} \to S/U_j \to 0,\text{ for }1\leq j\leq mt-1.
\end{equation}
Note that, according to Lemma \ref{inmt}, we have that
\begin{equation}\label{cia2}
(L_0,x_j)\cong (I_{n-1,m}^t,x_n)\text{ for all }1\leq j\leq mt-1,
\end{equation}
where $I_{n-1,m}^t \subset S'=K[x_1,\ldots,x_{n-1}]$ and the isomorphism is given by the circular permutation of variables which send $j$ to $n$.
On the other hand, we have
$$U_j=(L_{j-1},x_j)=((J_{n,m}^t:x_1\ldots x_{j-1}),x_j)=$$
\begin{equation}\label{cia3}
=((J_{n,m}^t,x_j):x_1\cdots x_{j-1})\cong (I_{n-1,m}^t:x_{n-j+1}\cdots x_{n-1}),
\end{equation}
for all $1\leq j\leq mt-1$. From \eqref{cia2}, \eqref{cia3}, Lemma \ref{lem} and Theorem \ref{depth} it follows that
\begin{align}\label{nsa}
& \depth(S/U_j)\geq \depth(S'/I_{n-1,m}^t)=\varphi(n-1,m,t)\text{ and }\\
& \sdepth(S/U_j)\geq \sdepth(S'/I_{n-1,m}^t)\geq \varphi(n-1,m,t).\label{nsaa}
\end{align}
Also, from Lemma \ref{lem}, we have that
\begin{align}\label{nsa2}
& \depth(S/L_0)\leq \depth(S/L_1)\leq \cdots \leq \depth(S/L_{tm-1})\text{ and }\\
& \sdepth(S/L_0)\leq \sdepth(S/L_1)\leq \cdots \leq \sdepth(S/L_{tm-1}).\label{nsaa2}
\end{align}
We consider three cases:
\begin{enumerate}
\item[(i)] If $n=mt$ then, according to Lemma \ref{lemoasa}(2), it follows that
$$
\sdepth(S/L_{tm-1})=\depth(S/L_{tm-1})=n-t.
$$
\item[(ii)] If $mt<n\leq m(t+1)$ then, according to Lemma \ref{lemoasa}(3), it follows that
$$
\sdepth(S/L_{tm-1})=\depth(S/L_{tm-1})=n-t-1.
$$
\item[(iii)] If $n>m(t+1)$ then, according to Lemma \ref{lemoasa}(4), it follows that
$$
\sdepth(S/L_{tm-1})\geq \depth(S/L_{tm-1})=(m-1)t+\varphi(n-mt-1,m,t).
$$
\end{enumerate}
In all of the above cases (i), (ii) and (iii), it is easy to see that the following inequalities hold:
\begin{equation}\label{nsa3}
\sdepth(S/L_{tm-1})\geq \depth(S/L_{tm-1})\geq \varphi(n-1,m,t).
\end{equation}
From \eqref{nsa}, \eqref{nsaa}, \eqref{nsa2}, \eqref{nsaa2}, \eqref{nsa3}, the short exact sequences \eqref{cia1}, Lemma \ref{l11} and Lemma \ref{asia} it follows that
$$\depth(S/J_{n,m}^t)\geq \varphi(n-1,m,t)\text{ and }\sdepth(S/J_{n,m}^t)\geq \varphi(n-1,m,t).$$
Also, from Theorem \ref{t3}(1), it follows that $\depth(S/J_{n,m}^t)\leq \varphi(n-1,m,t)+1$.
Thus, we complete the proof of (2) and (3).
\end{proof}

\begin{obs}\rm
Note that, in the case (1) of Theorem \ref{t5} we have that $d=\gcd(n,m)=1$. 
However, the result is stronger than the result from Theorem \ref{t2}(1), since $t_0=n-1$ is 
larger than $t=\frac{n-1}{m}$.
\end{obs}

\subsection*{Aknowledgments} 

We gratefully acknowledge the use of the computer algebra system Cocoa (\cite{cocoa}) for our experiments.

Mircea Cimpoea\c s was supported by a grant of the Ministry of Research, Innovation and Digitization, CNCS - UEFISCDI, 
project number PN-III-P1-1.1-TE-2021-1633, within PNCDI III.





\end{document}